\documentclass[12pt]{article}
\usepackage{graphicx,psfrag,epsf}
\usepackage{enumerate}
\usepackage{amsmath,amsthm,amsfonts,epsfig,amssymb,natbib,eucal,eufrak}

\usepackage{booktabs}
\usepackage{multirow}
\usepackage{siunitx}
\usepackage{qtree}

\usepackage{float}
\restylefloat{table}
\usepackage{color}

\usepackage{pdfpages}
\usepackage{wrapfig}

\usepackage{comment}

\usepackage{tikz}
\usepackage{csquotes}

\usepackage{hyperref}
\newcommand{\MRhref}[1]{\href{https://mathscinet.ams.org/mathscinet-getitem?mr=#1}{MR#1}}

\newtheorem{thm}{Theorem}
\newtheorem{claim}{Claim}

\newtheorem{lem}{Lemma}

\newtheorem{cor}{Corollary}
\newtheorem{pro}{Proposition}

\newcommand{\blind}{0}

\newcommand{\Pbb}{\mathbb P}
\newcommand{\Ebb}{\mathbb E}
\newcommand{\1}{\mathbf 1}

\newif\ifshowrefs
\showrefsfalse

\begin{document}

\def\spacingset#1{\renewcommand{\baselinestretch}%
{#1}\small\normalsize} \spacingset{1}


\if0\blind
{
  \title{\bf Distinct Extreme Scores in Random Round-Robin Tournaments}

  \author{
  Yaakov Malinovsky
    \thanks{email: yaakovm@umbc.edu}
   \\
    Department of Mathematics and Statistics\\ University of Maryland, Baltimore County, Baltimore, MD
    21250, USA\\
}
  \maketitle
} \fi

\if1\blind
{
  \bigskip
  \bigskip
  \bigskip
  \begin{center}
    {\LARGE\bf Title}
\end{center}
  \medskip
} \fi

\begin{abstract}
We consider a general round-robin tournament model with equally strong players, where $X_{ij}$ denotes the score of player $i$ against player $j$. We assume that $X_{ij}$ takes values in a countable subset of $[0,1]$ and satisfies $X_{ij}+X_{ji}=1$.
We prove that if $k(n)\to\infty$ as $n\to\infty$ and
$
\frac{k(n)^2\sqrt{\log(n/k(n))}}{\sqrt n}\to 0,
$
then, with probability tending to one, the largest $k(n)$ scores are all distinct. In particular, this holds whenever
$k(n)=o\!\Bigl(\bigl(n/\log n\bigr)^{1/4}\Bigr).$
By symmetry, the same conclusion also holds for the lowest $k(n)$ scores.
The obtained scale coincides with the one arising in classical problems on distinct extreme degrees in Erdős–Rényi random graphs, despite the fundamentally different dependence structure.
This suggests that distinctness of extreme values may persist under broad classes of models exhibiting weak dependence.
\end{abstract}

\noindent%

\medskip
\noindent


\noindent%

\spacingset{1.45} 
\section{Introduction and Problem Statement}
In a round-robin tournament, each of $n$ players competes against each of the other $n-1$ players. When player
$i$ plays against player $j$, player $i$'s reward is a random variable $X_{ij}$. Let
\begin{equation}
\label{eq:not}
s_i(n)=\sum_{j=1,\, j\neq i}^{n} X_{ij}
\end{equation}
denote the score of player $i$, $1\le i\le n$, after playing against all other $n-1$ players.
We refer to $\bigl(s_1(n),s_2(n),\ldots,s_n(n)\bigr)$ as the score sequence of the tournament.

{Model $M_{[0,\,1]}$}: Let $D \subset [0, 1]$ be a countable set of all possible values of $X_{ij}$.
Every element of $D$ has positive probability.
For $i\neq j$, $X_{ij}+X_{ji}=1$, $X_{ij}, X_{ji} \in D$. This together with the assumption that all players are
equally strong (i.e., $X_{ij}$ are identically distributed), implies
\begin{equation}
\label{eq:symD}
0<\mathbb{P}\left(X_{ij}=a\right)=\mathbb{P}\left(X_{ij}=1-a\right)\,\,\,\, \text{for any}\,\,\,\, a\in D.
\end{equation}
We assume that $|D|>1$. For each unordered pair $\{i,j\}$,
the variables $X_{ij}$ and $X_{ji}$ are dependent through the relation
$X_{ij}+X_{ji}=1$. Apart from this dependence within a match, outcomes
of different matches are independent; that is, the pairs
$\bigl(X_{ij},X_{ji}\bigr)$, $1\le i<j\le n$, are mutually independent.

The classical round-robin tournament corresponds to
$D=\{0,1\},$
which we denote by Model $M_1$. The classical chess model corresponds to $
D=\left\{0,\frac12,1\right\}$ (Model $M_2$). More generally, we consider finite grid models $M_k=\left\{0, \frac{1}{k}, \frac{2}{k},\ldots, \frac{k}{k}\right\}$. 

The problem of measuring players' strengths in chess tournaments via paired comparisons has a long history, beginning with \cite{L1895}. In \cite{Z1929}, a model for winning probabilities was proposed, and maximum likelihood estimators were derived and applied to the analysis of a chess tournament. The combinatorial investigation of Models $M_1$ and $M_2$ was initiated in \cite{M1923, M1923b} and continues to attract attention; see, for example, \cite{ABKKW2006, H2026, BDK2026}.  
An introduction to the combinatorial and probabilistic aspects of tournaments, along with extensive references to earlier works, is provided in the classical monograph \cite{Moon1968}; see also the corrected 2013 version available via \href{https://www.gutenberg.org/ebooks/42833}{Project Gutenberg}.

Round-robin tournaments provide a natural probabilistic framework for paired-comparison models and related statistical inference, beginning with the works of \citep{Z1929, KB1940}, discussed extensively in the monograph \cite{D1988}, and continuing to attract considerable attention to this day; see, for example, \cite{BF2018, N2023}.
The asymptotic distribution of extreme scores in Model $M_{[0,1]}$ was recently derived in \cite{M2026}, where further connections and an extensive bibliography are provided.
A related problem was considered recently in \cite{DR2026}.

Let $r_n(M)$ denote the probability that the tournament with $n$ players has a unique player with maximum score under model $M$.
\cite{E1967} stated, without proof, that in a classical round-robin tournament (Model $M_1$), the probability that there is a unique
player with the maximum score tends to $1$ as $n$ tends to infinity, i.e.,
\begin{equation}
\label{eq:C1}
{\displaystyle \lim_{n\rightarrow \infty}r_n(M_1)=1.}
\end{equation}
In a survey paper, \cite{G1984} noted that Epstein's claim \eqref{eq:C1} was still open at the time.
In a recent publication, \cite{MM2024} proved \eqref{eq:C1} using a method introduced in \cite{EW1977}.
More recently, \cite{AM2026} extended this result to the model $M_{[0,1]}$.

We denote the vector of ordered scores by
$$
s_{(1)}(n)\le s_{(2)}(n)\le\cdots\le s_{(n)}(n).
$$

Define the event
\begin{equation}
\label{eq:U}
U_{n,k}
=
\Bigl\{
\text{the $k$ largest scores } s_{(n-k+1)}(n), s_{(n-k+2)}(n), \ldots, s_{(n)}(n)
\text{ are pairwise distinct}
\Bigr\}.
\end{equation}

In this work, we obtain conditions on $k(n)$ under which
\begin{equation*}
\lim_{n\to\infty}\mathbb{P}\!\left(U_{n,k(n)}\right)=1.
\end{equation*}

The tournament can be represented by a complete directed graph whose vertices correspond to the players. For each pair of distinct vertices $i$ and $j$, the directed edges $(i,j)$ and $(j,i)$ are assigned weights $X_{ij}$ and $X_{ji}$ satisfying
$
X_{ij}+X_{ji}=1.
$
Thus, each match distributes one unit of score between the two players.
It is important to note that related questions have also been studied in the theory of random graphs. 
In the Erd\H{o}s--R\'enyi random graph model \citep{G1959, ER1959, ER1960, B2001}, there are $n$ vertices, and each unordered pair of distinct vertices $i$ and $j$ is connected by an undirected edge independently with probability $p$ (which may depend on $n$).
A related problem concerning the distinctness of extreme vertex degrees was posed in \cite{ER1968} and studied by \cite{B1981}. An interesting application of the obtained results on the gaps between extreme degrees was found in the construction of a simple algorithm for the graph isomorphism problem \cite{BES1980,B2001}.

This direction remains an active area of research to this day; see, for example, \cite{JKLP1993} and \cite{IZ2025}.
Although the dependence structure in random tournaments is fundamentally different and exhibits negative dependence, the behavior of the extreme scores turns out to be strikingly similar to that in random graphs. 
In contrast to the random graph setting, our proof relies on different methods, combining large deviations, concentration inequalities, and the negative dependence structure of tournament scores, particularly negative association.

\section{Main Result}
\begin{thm}
\label{thm1}
If $k(n)\to\infty$ as $n\to\infty$ and
$$
\frac{k(n)^2\sqrt{\log(n/k(n))}}{\sqrt n}\to 0,
$$
then
\begin{equation*}
\lim_{n\to\infty}\mathbb{P}\bigl(U_{n,k(n)}\bigr)=1.
\end{equation*}
\end{thm}
\medskip

\begin{cor}
If
\[
k(n)=o\!\left(\frac{n^{1/4}}
{(\log n)^{1/4}}\right),
\]
then the condition of Theorem~\ref{thm1} holds.
\end{cor}
\medskip

A related property concerning the set of the smallest scores can be obtained due to the symmetry of the outcome distribution assumed in \eqref{eq:symD}.
Recall the definition $U_{n,k}$ in \eqref{eq:U} and define
$$
\widetilde U_{n,k}
=\Bigl\{
\text{the $k$ smallest scores } s_{(1)}(n), s_{(2)}(n), \ldots, s_{(k)}(n)
\text{ are pairwise distinct}
\Bigr\}.
$$

\begin{pro}
\label{cor:sym}
For each $k\in \left\{1, 2, \ldots, n\right\}$,
$$
\Pbb\big(U_{n,k}\big)
=
\Pbb\big(\widetilde U_{n,k}\big).
$$
\end{pro}

\begin{proof}(Proposition \ref{cor:sym} )
Let $\widetilde T_n$ be the tournament obtained from $T_n$ by reversing the orientation of every match. Then $\widetilde T_n$ has the same distribution as $T_n$, and its scores satisfy
$$
\widetilde s_i=(n-1)-s_i, \qquad 1\le i\le n.
$$

Hence, the event that the top $k$ scores are pairwise distinct in $T_n$ coincides with the event that the bottom $k$ scores are pairwise distinct in $\widetilde T_n$. Since $\widetilde T_n \stackrel{d}{=} T_n$, the conclusion follows.
\end{proof}

\begin{proof}(Theorem \ref{thm1})

To prove the theorem, we require the following notation and three lemmas.
\medskip

From the symmetry $X_{ij} \stackrel{d}{=} 1-X_{ij}$  it follows that
$$\mu=\mathbb{E}(X_{ij})=\frac{1}{2}.$$
Since $X_{ij}\in[0,1]$, we also have $$\sigma=\bigl(\mathrm{Var}(X_{ij})\bigr)^{1/2}\le \frac{1}{2}.$$
Furthermore, let
\begin{align}
\label{eq:tn}
t_{n, k}=(n-1)\mu+x_{n, k} (n-1)^{1/2}\sigma,
\end{align}
where $x_{n,k}$ will later be chosen so that the expected number of scores exceeding $t_{n,k}$ is of order $k$.

Let
$$
I_j(t)=\1\{s_j(n)>t\}, \qquad Z_t=\sum_{j=1}^n I_j(t).
$$

Define
$$
W_n(t)=\sum_{1\le v<u\le n}\1\{t<s_u(n)=s_v(n)\}.
$$
\medskip

\begin{lem}
\label{pro:1}
Let $k(n)\to\infty$ with $k(n)=o(n)$. Then, for any fixed $\delta\in(0,1)$, one can choose $x_{n,k}\to\infty$ satisfying 
$
x_{n,k}=o(n^{1/6})$ and
\begin{equation}
\label{eq:Cond}
x_{n,k}^2=2\log\!\bigl(n/k(n)\bigr)+O\!\bigl(\log\log(n/k(n))\bigr)
\end{equation}
such that
$$
n\,\mathbb{P}\!\left(s_1(n)>t_{n,k}\right)\sim (1+\delta)\,k(n).
$$
\end{lem}

\begin{proof}
See Appendix \ref{app:1}.
\end{proof}
\bigskip

\begin{lem}
\label{pro:2}
Suppose that $k(n)\to\infty$ with $k(n)=o(n)$, and that
$x_{n,k}$ is chosen as in Lemma~\ref{pro:1}, for some fixed
$\delta\in(0,1)$. Then, for $t_{n,k}$ defined in \eqref{eq:tn},
$$
\lim_{n \to \infty} \Pbb\!\left(Z_{t_{n,k}}<k(n)\right)=0.
$$
\end{lem}

\begin{proof}
See Appendix \ref{app:2}.
\end{proof}
\bigskip

\begin{lem}
\label{pro:3}
Assume Model $M_{[0,1]}$ with countable support $D\subset[0,1]$, $|D|>1$.
Let $k(n)\to\infty$, $k(n)=o(n)$, and suppose that $x_{n,k}\to\infty$ with
$x_{n,k}=o(n^{1/6})$ satisfy \eqref{eq:Cond}. Let $t_{n,k}$ be defined in \eqref{eq:tn}.
Then there exists a constant $C>0$ such that for all sufficiently large $n$,
\begin{equation*}
\Ebb\!\left(W_n(t_{n,k})\right)
\le
C\,\frac{k(n)^2\sqrt{\log(n/k(n))}}{\sqrt n}.
\end{equation*}
\end{lem}

\begin{proof}
See Appendix \ref{app:3}.
\end{proof}
\bigskip

Therefore, noting that $W_n(t_{n,k})\ge \1\{W_n(t_{n,k})\ge 1\}\ge 0$, we obtain from Lemma \ref{pro:3} that
if $k(n)\rightarrow \infty$, as $n\rightarrow \infty$, so that $\frac{k(n)^2\sqrt{\log(n/k(n))}}{\sqrt n}\rightarrow 0$, then
\begin{equation}
\label{eq:pronew}
\lim_{n \to \infty} \mathbb{P}\bigl(W_n(t_{n, k}) = 0\bigr) = 1.
\end{equation}

Define
\begin{equation}
\label{eq:int}
G=\{Z_{t_{n, k}}\ge k\}\cap\{W_n(t_{n, k})=0\}.
\end{equation}

Since $G\subseteq U_{n,k}$, it follows from \eqref{eq:int} that
\begin{equation}
\label{eq:Ine1}
0\le \Pbb\!\left(U^{C}_{n,k}\right)
\le \Pbb\!\left(G^{C}\right)
\le \Pbb\!\left(Z_{t_{n, k}}<k\right)+\Pbb\!\left(W_n(t_{n, k})>0\right).
\end{equation}
Combining \eqref{eq:Ine1} with Lemma~\ref{pro:2} and~\eqref{eq:pronew}
completes the proof of Theorem \ref{thm1}.
\end{proof}
\smallskip

\appendix

\section{Proofs}

\subsection{Proof of Lemma \ref{pro:1}}
\label{app:1}

\begin{proof}
If $x_{n, k}\rightarrow \infty$, as $n\rightarrow \infty$, so that $x_{n, k}=o(n^{1/6})$ and $X_{ij} \in [0, 1]$, then it follows
from \cite[p. 553, Theorem 3]{F1971} that
\begin{align}
\label{eq:LD1}
&
\mathbb{P}(s_1(n)>t_{n, k})\sim 1-\Phi(x_{n, k}),
\end{align}
where $\Phi()$ denotes the standard normal distribution function.
Also, since $x_{n, k}\rightarrow \infty$ as $n\rightarrow \infty$ (see for example, \cite[p. 175, Lemma 2]{F1968}),
\begin{equation}
\label{eq:LD2}
1-\Phi(x_{n, k})\sim \frac{1}{x_{n, k}}\varphi(x_{n, k}),
\end{equation}
where $\varphi()$ is the PDF of a standard normal random variable.

From \eqref{eq:LD1} and \eqref{eq:LD2}, it follows that
if $x_{n, k}\to\infty$ as $n\to\infty$
and $x_{n, k}=o(n^{1/6})$, then
\begin{equation}
\label{eq:LD3}
\mathbb{P}(s_1(n)>t_{n, k})\sim \frac{1}{x_{n, k}}\,\varphi(x_{n, k}).
\end{equation}

Let $k=k(n)$ and $x=x_{n,k}>0$ satisfy
\begin{equation}
\label{eq:sat}
\frac{n}{\sqrt{2\pi}}\,
\frac{1}{x}\,e^{-x^2/2}
=(1+\delta)\,k .
\end{equation}
Equivalently,
$$
e^{-x^2/2}
=
\frac{(1+\delta)\,k\,x\sqrt{2\pi}}{n}.
$$
Therefore,
$$
\frac{x^2}{2}
=
\log\frac{n}{k}
-\log x
-\frac12\log(2\pi)
-\log(1+\delta),
$$
and hence
\begin{equation}
\label{eq:xx}
x^2
=
2\log(n/k)
+O(\log\log(n/k)).
\end{equation}
Assuming that $k(n)=o(n)$, it follows from \eqref{eq:xx} that  $x_{n,k}=o(n^{1/6})$ and appealing to
\eqref{eq:LD3}, \eqref{eq:sat} and \eqref{eq:xx} we complete the proof of Lemma \ref{pro:1}.
\end{proof}

\subsection{Proof of Lemma 2}
\label{app:2}
\begin{proof}
Recall that
$
Z_t=\sum_{j=1}^n \1\{s_j(n)>t\},
$
that $t_{n,k}$ is defined in \eqref{eq:tn}, and that the scores
$s_1(n),\ldots,s_n(n)$ are identically distributed.
Suppose that $k(n)\to\infty$ with $k(n)=o(n)$ and that $x_{n,k}\to\infty$ with
$x_{n,k}=o(n^{1/6})$. Assume further that $k(n)$ and $x_{n,k}$ satisfy
\eqref{eq:Cond}. Then, from Lemma~\ref{pro:1}, it follows that for any fixed
$\delta\in(0,1)$,
\begin{equation}
\label{eq:AZ}
\Ebb\!\left(Z_{{t_{n, k}}}\right)=n\Pbb\!\left(s_1(n)>t_{n, k}\right)\sim (1+\delta)\,k(n)\to\infty .
\end{equation}
Let
$$
\eta=\frac{\delta}{1+\delta}\in(0,1).
$$
Appealing to \eqref{eq:AZ}, we obtain
$$
(1-\eta)\Ebb\!\left(Z_{{t_{n, k}}}\right)
=\frac{\Ebb\!\left(Z_{{t_{n, k}}}\right)}{1+\delta}
\sim k(n),
$$
and therefore, for all sufficiently large $n$,
$$
k(n)\le (1-\eta/2)\Ebb\!\left(Z_{{t_{n, k}}}\right).
$$
Hence
\begin{align}
\label{eq:inZ}
&
0\le \Pbb\!\left(Z_{{t_{n, k}}}<k(n)\right)
\le
\Pbb\!\left(Z_{{t_{n, k}}}<(1-\eta/2)\Ebb\!\left(Z_{{t_{n, k}}}\right)\right)\\
&
=
\Pbb\!\left(Z_{{t_{n, k}}}-\Ebb\!\left(Z_{{t_{n, k}}}\right)
<-(\eta/2)\Ebb\!\left(Z_{{t_{n, k}}}\right)\right) \nonumber
\le
\Pbb\!\left(\big|Z_{{t_{n, k}}}-\Ebb\!\left(Z_{{t_{n, k}}}\right)\big|
>(\eta/2)\Ebb\!\left(Z_{{t_{n, k}}}\right)\right)\\ \nonumber
&
\le
\frac{Var(Z_{{t_{n, k}}})}
{(\eta/2)^2\big(\Ebb\!\left(Z_{{t_{n, k}}}\right)\big)^2}
\le
\frac{\Ebb\!\left(Z_{{t_{n, k}}}\right)}
{(\eta/2)^2\big(\Ebb\!\left(Z_{{t_{n, k}}}\right)\big)^2}
=
\frac{4}{\eta^2\,\Ebb\!\left(Z_{{t_{n, k}}}\right)},
\end{align}
where the first inequality follows from adding a nonnegative term,
the second from Chebyshev's inequality, and the third from the facts that
$Var(I_1({t_{n, k}}))\le \Ebb(I_1({t_{n, k}}))$ and
$$\operatorname {Cov}(I_1({t_{n, k}}),I_2({t_{n, k}}))\leq 0.$$ 
The covariance bound
$
\operatorname{Cov}(I_1(t_{n,k}),I_2(t_{n,k}))\le 0
$
holds because the vector \( (I_1(t_{n,k}),\ldots,I_n(t_{n,k})) \) is negatively associated as proved in \cite{MR2022}.
Combining \eqref{eq:AZ} with \eqref{eq:inZ} under assumptions of the lemma, we obtain
$$
\lim_{n \to \infty} \Pbb\!\left(Z_{t_{n, k}}<k(n)\right)=0.
$$
\end{proof}

\subsection{Proof of Lemma 3}
\label{app:3}
\begin{proof}
Recall that
$
W_n(t)=\sum_{1\le v<u\le n}\mathbf 1\{t<s_u(n)=s_v(n)\}.
$
Also, in view of the definition of the scores in \eqref{eq:not}, let $s_u(n-1)$ denote the total score of player $u$ in the tournament obtained after removing one opponent, that is, after $n-2$ matches.
Assume from now that $n\ge3$, and let
$$
A_n=\{x:\ t_{n,k}-1<x\le n-2 \ \text{and $x$ is an attainable value of } s_u(n-1)\}.
$$

By conditioning on the outcome of the match between players $u$ and $v$, \cite[Proposition 2]{AM2026} established the following claim (see also \cite{MM2024} for the corresponding result in model $M_1$). In the original proposition, $t_{n,k}$ was replaced by an analogous quantity that does not depend on $k$; however, this modification has no effect on the proof.
\begin{claim}
\begin{equation}
\label{eq:AM1}
\Ebb\!\left(W_n(t_{n,k})\right)
\le
\frac{n(n-1)}{2}\,\Pbb\!\left(s_1(n-1)>t_{n,k}-1\right)\,
\sup_{x\in A_n}\Pbb\!\left(s_1(n-1)=x\right).
\end{equation}
\end{claim}

We next bound the two probability terms appearing on the right-hand side of \eqref{eq:AM1}, assuming Model $M_{[0,1]}$, $k(n)\to\infty$, $k(n)=o(n)$, and $x_{n,k}\to\infty$ with $x_{n,k}=o(n^{1/6})$, where $x_{n,k}$ satisfies \eqref{eq:Cond}.

Define
\begin{equation}
\label{eq:L}
l_{n,k}
=
t_{n,k}-1
=
(n-2)\mu+y_{n,k}(n-2)^{1/2}\sigma,
\end{equation}
where
$$
y_{n,k}
=
\frac{t_{n,k}-1-(n-2)\mu}
{\sigma (n-2)^{1/2}}.
$$
Since $\mu=1/2$ and recalling the definition of $t_{n,k}$ in \eqref{eq:tn}, we have
$$
y_{n,k}
=
x_{n,k}\sqrt{\frac{n-1}{n-2}}
-
\frac{1}{2\sigma\sqrt{n-2}}
=
x_{n,k}+O(n^{-1/2}),
$$
and therefore
\begin{equation}
\label{eq:xy}
y_{n,k}^2=x_{n,k}^2+o(1).
\end{equation}

The first bound in the RHS of \eqref{eq:AM1} follows from the proof of Lemma~\ref{pro:1} with $n$ replaced by $n-1$ and $x_{n,k}$ replaced by $y_{n,k}$ and appealing to equations \eqref{eq:L} and \eqref{eq:xy}.
Thus,
\begin{equation}
\label{eq:b1}
\Pbb\!\left(s_v(n-1)>t_{n,k}-1\right)
\le C_1 \frac{k(n)}{n}
\end{equation}
for some constant $C_1>0$.
The second bound is stated in the following claim. Although the proof follows the same general ideas as those used in \cite{AM2026}, the present setting is more general since the threshold depends on $k(n)$. For the sake of completeness, we provide the details of the proof in the Appendix \ref{app:4}.
\begin{claim}
\label{cl}
\begin{equation}
\label{eq:b2}
\sup_{x\in A_n}\Pbb(s_1(n-1)=x)
\le C_2 \frac{k(n)\sqrt{\log(n/k(n))}}{n^{3/2}}
\end{equation}
for some constant $C_2>0$.
\end{claim}
\begin{proof}
See appendix \ref{app:4}.
\end{proof}

Substituting the bounds \eqref{eq:b1} and \eqref{eq:b2} into \eqref{eq:AM1} yields
$$
\Ebb(W_n(t_{n,k}))
\le
C\,\frac{k(n)^2\sqrt{\log(n/k(n))}}{\sqrt n}
$$
for some constant $C>0$.
\end{proof}

\subsection{Proof of Claim \ref{cl}}
\label{app:4}
\begin{proof}
Recall definition $t_{n, k}$ in \eqref{eq:tn} and that we assume $x_{n,k}=o(n^{1/6})$, where $x_{n,k}$ satisfies \eqref{eq:Cond}.

Let $m = n - 2$ and $S_m = s_1(n-1)$.

For any $\theta\in\mathbb{R}$, define the tilted random variable
\begin{equation}
\label{eq:til}
\Pr_{\theta}(S_m=x)
=\frac{\mathbb{E}\!\left(e^{\theta S_m}\mathbf{1}_{\{S_m=x\}}\right)}
{\mathbb{E}\!\left(e^{\theta S_m}\right)}.
\end{equation}

By the definition \eqref{eq:til}, we have

\begin{align}
\label{eq:b}
&
\mathbb{P}(S_m=x)=e^{-\theta x}\,\mathbb{E}\!\left(e^{\theta S_m}\mathbf{1}_{\{S_m=x\}}\right)=
e^{-\theta x}\,\mathbb{E}\!\left(e^{\theta S_m}\right)\,
   \frac{\mathbb{E}\!\left(e^{\theta S_m}\mathbf{1}_{\{S_m=x\}}\right)}
        {\mathbb{E}\!\left(e^{\theta S_m}\right)}\nonumber\\[4pt]
&=e^{-\theta x}\,\mathbb{E}\!\left(e^{\theta S_m}\right)\,
   \Pr_\theta(S_m=x).
\end{align}

For $\theta>0$ and $x>l_{n,k}$, it follows from \eqref{eq:L} and \eqref{eq:b} that
\begin{align}
\label{eq:new1}
\mathbb{P}(S_m=x)
= e^{-\theta x}\,\mathbb{E}\!\left(e^{\theta S_m}\right)\,
   \Pr_\theta(S_m=x)
\le e^{-\theta l_{n, k}}\,\mathbb{E}\!\left(e^{\theta S_m}\right)\Pr_\theta (S_m=x),
\end{align}
where the inequality follows because $e^{-\theta x}$ is non-increasing in $x$ for $\theta>0$ and $x>l_{n,k}$.

Therefore, for $\theta>0$, and for sufficiently large $n$
\begin{equation}
\label{eq:sup}
\sup_{x>t_{n, k}-1}\mathbb{P}(S_m=x)
=
\sup_{x>l_{n, k}}\mathbb{P}(S_m=x)
\le
e^{-\theta l_{n, k}}\,\mathbb{E}\!\left(e^{\theta S_m}\right)
\sup_{x>l_{n, k}}\Pr_\theta(S_m=x).
\end{equation}

Equation (18) of \cite{AM2026} implies that, under Model $M_{[0,1]}$, for all sufficiently large $n$ and choosing $\theta=\frac{y_{n,k}}{\sigma\sqrt{m}}$, and appealing to \eqref{eq:xy}, 
\begin{align}
\label{eq:first}
e^{-\theta l_{n, k}} \mathbb{E}\!\left(e^{\theta S_m}\right)
\le e^{\Big(-\frac{x_{n, k}^2}{2} + o(1)\Big)}.
\end{align}

Also using Kolmogorov's inequality on the rate of decrease of Lévy's concentration function for the sum of independent discrete random variables obtained by \cite{R1961} (see also \cite{P1975}, pages 56 and 319), 
under Model $M_{[0,1]}$,
\begin{align}
&
\label{eq:second}
\sup_{x > l_{n, k}} \Pr_\theta(S_m = x)
\le \frac{c}{\sqrt{m}},
\end{align}
for some constant $c>0$.

Combining \eqref{eq:sup}, \eqref{eq:first}, and \eqref{eq:second} we obtain for sufficiently large $n$,
\begin{equation}
\label{eq:AM2}
\sup_{x \in A_n} \mathbb{P}(s_1(n-1) = x)\leq O\left(\frac{1}{n^{1/2}}e^{\Big(-\frac{x_{n, k}^2}{2} + o(1)\Big)}\right).
\end{equation}

Now, \eqref{eq:AM2} and \eqref{eq:Cond} imply
$$
\sup_{x\in A_n}\Pbb(s_1(n-1)=x)
\le C_2 \frac{k(n)\sqrt{\log(n/k(n))}}{n^{3/2}},
$$
for some constant $C_2>0$.

\end{proof}

\section*{Acknowledgements}
I would like to thank Noga Alon for valuable discussions, John W.~Moon and Miklós Simonovits for helpful comments and suggestions. I thank Boris Alemi for carrying out the computations on the department’s high-performance computer. This research was supported in part by BSF grant 2020063.



\end{document}